\documentclass[11pt]{amsart}

\usepackage{amsfonts, amstext, amsmath, amsthm, amscd, amssymb, upgreek}
\usepackage{graphicx, color,  subfigure, wrapfig, overpic}
\usepackage[all,cmtip]{xy} 
\usepackage{tikz}

\let\oldmarginpar\marginpar
\renewcommand\marginpar[1]{\oldmarginpar[\raggedleft\footnotesize #1]%
{\raggedright\footnotesize #1}}

\AtBeginDocument{
   \def\MR#1{}
}

\newcommand{\Z}{\mathbb{Z}}

\newcommand{\W}{\mathcal W}
\newcommand{\vol}{{\rm vol}}

\newcommand{\voct}{{v_{\rm oct}}}
\newcommand{\vtet}{{v_{\rm tet}}}
\renewcommand{\L}{\mathcal L}

\newcommand{\toF}{{\overset{F}{\longrightarrow}}}
\newcommand{\K}{\upkappa}
\newcommand{\In}{\text{in}}
\newcommand{\out}{\text{out}}
\newcommand{\E}{\mathcal E}

\theoremstyle{plain}
\newtheorem{theorem}{Theorem}[section]
\newtheorem{corollary}[theorem]{Corollary}
\newtheorem{lemma}[theorem]{Lemma}
\newtheorem{prop}[theorem]{Proposition}

\newtheorem*{namedtheorem}{\theoremname}
\newcommand{\theoremname}{testing}

\theoremstyle{definition}
\newtheorem{define}[theorem]{Definition}

\title[Determinant density and biperiodic alternating links]{Determinant density and\\ biperiodic alternating links}
\author[A.\ Champanerkar]{Abhijit Champanerkar}
\address{Department of Mathematics, College of Staten Island \& The Graduate Center, City University of New York, New York, NY}
\email{abhijit@math.csi.cuny.edu}
\author[I. \ Kofman]{Ilya Kofman}
\address{Department of Mathematics, College of Staten Island \& The Graduate Center, City University of New York, New York, NY}
\email{ikofman@math.csi.cuny.edu}

\begin{document}

\begin{abstract}
Let $\L$ be any infinite biperiodic alternating link.  We show that
for any sequence of finite links that F{\o}lner converges almost
everywhere to $\L$, their determinant densities converge to the Mahler
measure of the 2--variable characteristic polynomial of the toroidal
dimer model on an associated biperiodic graph.
\end{abstract}

\maketitle

\section{Introduction}

The determinant of a knot is one of the oldest knot invariants that
can be directly computed from a knot diagram.  For any knot or link $K$,
$$ \det(K) = |\det(M + M^T)| = |H_1(\Sigma_2(K);\Z)| = |\Delta_K(-1)|= |V_K(-1)|, $$ 
where $M$ is any Seifert matrix of $K$, $\Sigma_2(K)$ is the $2$--fold
branched cover of $K$, $\Delta_K(t)$ is the Alexander polynomial and
$V_K(t)$ is the Jones polynomial of $K$ (see, e.g., \cite{lickorish}).

In \cite{ckp:gmax}, with Jessica Purcell, we studied the volume and
determinant density of alternating hyperbolic links approaching the
infinite square weave $\W$, the biperiodic alternating link shown in
Figure~\ref{fig:weave}(a).  The {\em volume density} of a hyperbolic
link $K$ with crossing number $c(K)$ is defined as $\vol(K)/c(K)$, and
the {\em determinant density} of $K$ is defined as
$2\pi\log\det(K)/c(K)$.  The volume density is known to be bounded by
the volume of the regular ideal octahedron, $\voct \approx 3.66386$,
and the same upper bound is conjectured for the determinant density.
With a suitable notion of convergence of diagrams, called here {\em
  F{\o}lner convergence almost everywhere}, $K_n\toF\W$ as in
Definition \ref{def:folner_converge} below, we proved:

\begin{theorem}\cite{ckp:gmax}\label{thm:gmax}
Let $K_n$ be any alternating hyperbolic link diagrams with no cycles of tangles such that $K_n\toF\W$.
Then $K_n$ is both geometrically and diagrammatically maximal:
$$ \lim_{n\to\infty}\frac{\vol(K_n)}{c(K_n)}= \lim_{n\to\infty}\frac{2\pi\log\det(K_n)}{c(K_n)}=\voct. $$
\end{theorem}

We define the volume density of $\W$ as $\vol(L)/c(L)$, where $L$ is
the finite toroidally alternating $\Z^2$--quotient link shown in
Figure~\ref{fig:weave}(b).  Here, $c(L)$ is the crossing number of the
reduced alternating projection of $L$ on the torus, which is minimal
by \cite{Adams:crossings}, and $\vol(L)=\vol(T^2\times I-L)$ is the
hyperbolic volume of its complement in the thickened torus $T^2\times
I$.  In this case, $c(L)=4$ and $\vol(L)=4\voct$ (see
\cite{ckp:gmax}). Hence, as $K_n\toF\W$, the volume densities of $K_n$
converge to the volume density of $\W$, $\vol(L)/c(L)=\voct$.

However, $\voct$ initially appears as a mysterious limit of the
determinant densities.  In \cite{ckp:gmax}, we proved that this limit
is the spanning tree entropy of the infinite square grid graph, which
is the planar projection graph of $\W$.  Just as in the case of volume
density, there is an analogous toroidal invariant of $\W$ that appears
as the limit of the determinant density.  This invariant is the
entropy of the toroidal dimer model on an associated biperiodic graph.

In this paper, we extend this diagrammatic result for $\W$ to {\em
  any} biperiodic alternating link $\L$.  We show that using the same
type of convergence of finite link diagrams as in Theorem
\ref{thm:gmax}, their determinant densities converge for any
biperiodic alternating link $\L$.  Moreover, we identify their limit
as the Mahler measure of the 2--variable polynomial arising from the
toroidal dimer model on a planar biperiodic graph.  Following the
definitions in the two subsections below, we present our main result
in Theorem~\ref{thm:main}.
\begin{figure}[h] 
\begin{tabular}{cc}
 \includegraphics[height=1.3in]{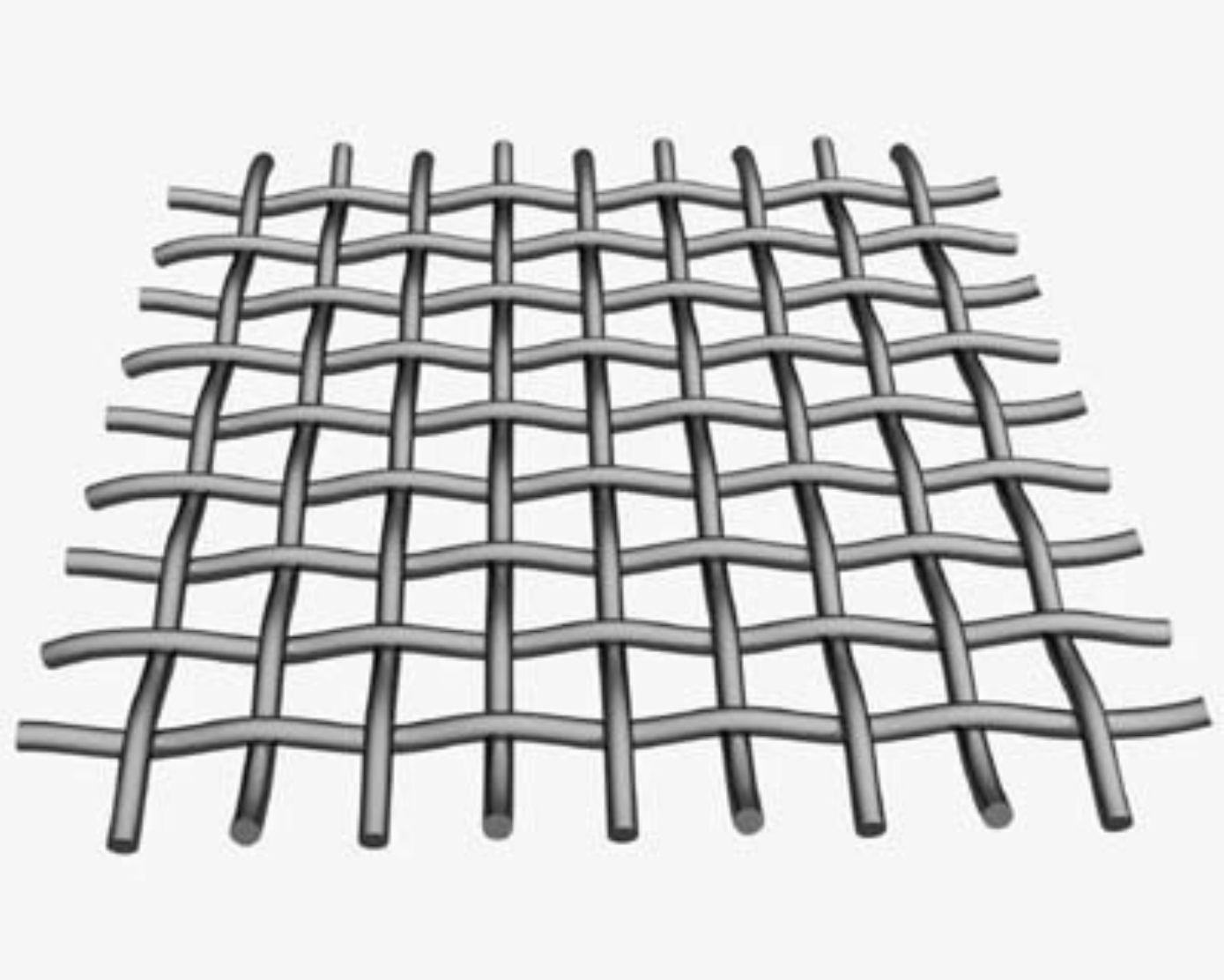} &  \quad \quad \quad \quad 
 \includegraphics[height=1in]{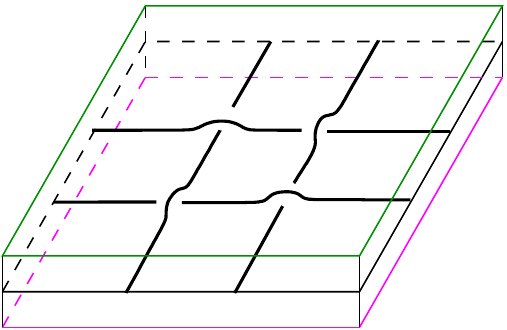}    \\
  (a) & \quad  \quad \quad \quad  (b)   \\
\end{tabular}
\caption{(a) Infinite weave $\W$. (b) Toroidally alternating quotient link $L$.}
\label{fig:weave}
\end{figure}

The main idea of the proof of Theorem~\ref{thm:main} is to relate the
limit of determinant densities of links approaching a biperiodic
alternating link $\L$ with the spanning tree entropy of a
corresponding planar graph $G_\L$, and to relate this, in turn, to the
entropy of the toroidal dimer model on a planar bipartite biperiodic
graph $G^b_\L$.  Thus, our main contribution is to bring some of the
beautiful new results from the asymptotics of the toroidal dimer model
to knot theory.

In Section~\ref{sec:definitions} we give several required definitions and then state our main theorem.
In Section~\ref{sec:dimer}, we discuss some aspects of the toroidal dimer model.
In Section~\ref{sec:examples}, we compute two examples that illustrate Theorem~\ref{thm:main}.
Finally, in Section~\ref{sec:proof}, we prove Theorem~\ref{thm:main}.

\section{Definitions and main result}\label{sec:definitions}

\subsection{F{\o}lner convergence of link diagrams}

The following notion of convergence of graphs is well known, but the
corresponding definition of convergence of link diagrams was
introduced in \cite{ckp:gmax}.  We say an infinite graph $G$ or an
infinite alternating link $\L$ is \emph{biperiodic} if it is invariant
under translations by a two-dimensional lattice $\Lambda$. 

\begin{define}\label{def:folner}
Let $G$ be any infinite graph. For any finite subgraph $G_n$,
the set $\partial G_n$ is the set of vertices of $G_n$ that share an edge
with a vertex not in $G_n$.  We let $|\cdot|$ denote the number of
vertices in a graph.  An exhaustive nested sequence of
connected subgraphs,
$\{G_n\subset G \mid \; G_n\subset G_{n+1},\; \bigcup_n G_n=G\}$, is a {\em F{\o}lner sequence} for $G$ if
\[ \lim_{n\to\infty}\frac{|\partial G_n|}{|G_n|}=0. \]
For a 
biperiodic planar graph $G$, we say $\{G_n\subset G\}$, is a {\em toroidal F{\o}lner sequence} for $G$
if it is a F{\o}lner sequence for $G$ such that $G_n\subset G\cap(n\Lambda)$.
\end{define}

\begin{define}\label{def:folner_converge_graphs}
Let $G$ be any biperiodic planar graph. We will say that a sequence of planar graphs $\Gamma_n$ {\em F{\o}lner converges almost everywhere} to $G$, denoted by $\Gamma_n\toF G$, if 
\begin{enumerate}
\item[($i$)] there are subgraphs $G_n\subset \Gamma_n$ that form a toroidal F{\o}lner sequence for $G$,
\item[($ii$)] $\lim\limits_{n\to\infty} |G_n|/|\Gamma_n| = 1$.
\end{enumerate}
\end{define}

\begin{define}\label{def:folner_converge}
We will say that a sequence of alternating links $K_n$ {\em F{\o}lner converges almost everywhere} to the biperiodic alternating link $\L$, denoted by $K_n\toF\L$,
if the respective projections graphs $\{G(K_n)\}$ and $G(\L)$ satisfy $G(K_n)\toF G(\L)$; i.e., the following two conditions are satisfied:
\begin{enumerate}
\item[($i$)] there are subgraphs $G_n\subset G(K_n)$ that form a toroidal F{\o}lner sequence for $G(\L)$,
\item[($ii$)] $\lim\limits_{n\to\infty} |G_n|/ c(K_n) = 1$.
\end{enumerate}
\end{define}

For example, below is a Celtic knot diagram that could be in a sequence $K_n\toF\W$:

\centerline{\includegraphics[height=0.9in]{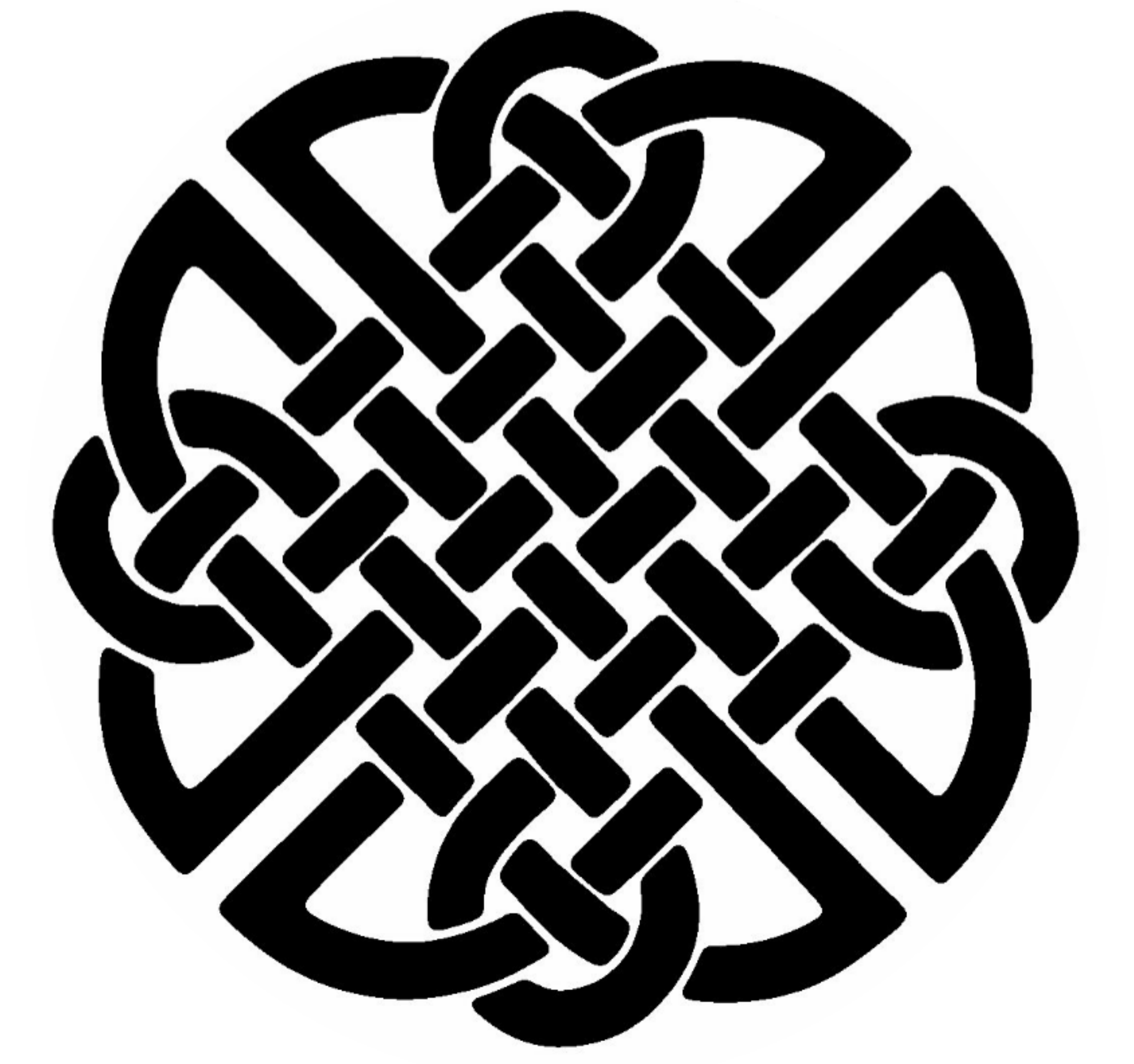}}

\noindent
As discussed in \cite{ckp:gmax, ckp:weaving}, weaving knots $W(p,q)$
with $p,q\to\infty$ provide examples of infinite sequences that F{\o}lner
converge almost everywhere to $\W$.  

\subsection{Mahler measure}
The Mahler measure of a polynomial $p(z)$ is defined as
$$ m(p(z)) =  \frac{1}{2\pi i}\int_{S^1} \log|p(z)|\;\frac{dz}{z} $$
It is a natural measure of complexity of polynomials that is additive under multiplication.
By Jensen's formula, $\displaystyle m(p(z)) = \sum_{i=1}^n\max \{\log|\alpha_i|,0 \}$, where $\alpha_1, \ldots, \alpha_n$ are roots of $p(z)$. 

The Mahler measure of a two-variable polynomial $p(z,w)$ is defined similarly:
 $$ m(p(z,w)) =\frac{1}{(2\pi i)^2}\int_{S^1\times S^1} \log|p(z,w)|\;\frac{dz}{z}\frac{dw}{w}.$$ 
 Unlike the one-variable case, two-variable Mahler measures are much
 harder to compute and exact values of $m(p(z,w))$ are known only for
 certain families of two-variable polynomials.

 Smyth's remarkable formula  below provided the
 first evidence of a deep relationship between the Mahler measure of
 two-variable polynomials and hyperbolic volume.  If $K$ is the
 figure-8 knot, $4_1$, then $\vol(K)=2\,\vtet$, where
 $\vtet\approx 1.01494$ is the hyperbolic volume of the regular ideal
 tetraheron.  Smyth \cite{smyth-mm-several} proved:
$$\displaystyle{2 \pi\, m(1+x+y) = \frac{3\sqrt{3}}{2}\,L(\chi_{{\text{\tiny -3}}},2) = \vol(K)}.$$
Later, Boyd and Rodriguez-Villegas \cite{BoydRV2002} related the
Mahler measure of $A$--polynomials of 1--cusped hyperbolic
3--manifolds to their hyperbolic volume.  
See the surveys  \cite{boyd-manifolds, smyth-survey} on the Mahler measure of one and two variable polynomials.

\subsection{Main result}

For any finite link diagram $K$, let $G(K)$ denote the projection
graph of $K$ as above, and let $G_K$ denote the Tait graph of $K$,
which is the planar checkerboard graph for which a vertex is assigned
to every shaded region and an edge to every crossing of $K$.  Using
the other checkerboard coloring yields the planar dual $G_K^*$.  Thus,
$e(G_K)=c(K)$.  Any alternating link $K$ is determined up to mirror
image by its Tait graph $G_K$.  Let $\tau(G)$ denote the number of
spanning trees of $G$.  By \cite{crowell}, for any connected reduced
alternating link diagram, 
$$\tau(G_K)=\det(K).$$

For a biperiodic alternating link $\L$, the projection
graph $G(\L)$ is biperiodic and can also be checkerboard
colored.  The two Tait graphs $G_{\L}$ and $G_{\L}^*$ are planar duals
and are both biperiodic.  We form the overlaid bipartite
graph $G_{\L}^b = G_{\L}\cup G_{\L}^*$ as follows: The black vertices
of $G^b$ are the vertices of $G_{\L}$ and of $G_{\L}^*$; the white
vertices of $G^b$ are points of intersection of their edges.  The
overlaid graph $G_{\L}^b$ is a biperiodic balanced bipartite graph;
i.e., the number of black vertices equals the number of white vertices
in a fundamental domain.  This makes it possible to define the
toroidal dimer model on $G_{\L}^b$, and in Section~\ref{sec:dimer} we
explain how to obtain the characteristic polynomial $p(z,w)$ of the
toroidal dimer model. The $\Lambda$--quotient link $L$ is the toroidal 
link $\L/\Lambda$. 

With these definitions, we can precisely state our main result: 

\begin{theorem}\label{thm:main}
Let $\L$ be any biperiodic alternating link, with toroidally alternating $\Lambda$--quotient link $L$. 
Let $p(z,w)$ be the characteristic polynomial of the toroidal dimer model on $G_{\L}^b$.
Then
$$ K_n\toF \L \quad \Longrightarrow \quad {\lim_{n\to\infty}\frac{\log\det(K_n)}{c(K_n)} = \frac{m(p(z,w))}{c(L)}}.$$ 
\end{theorem}

A similar limit for a particular closure of knots corresponding to
sublattices of $\L$ that grow in both directions is proved
independently by Silver and Williams \cite{sw-spanningtrees} using the
Laplacian polynomial.

We will call the right-hand side of the above equation the determinant
density of $\L$.

\begin{figure}[h] 
 \includegraphics[height=0.9in]{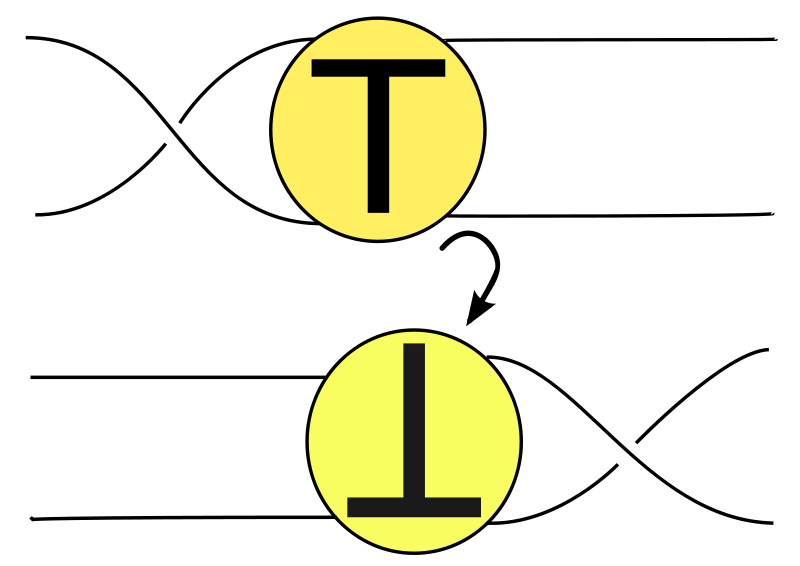}
\caption{Flype move on a link diagram (figure from \cite{flype_fig}).}
\label{fig:flype}
\end{figure}

A {\em flype}, shown in Figure \ref{fig:flype}, is a local move on
link diagrams that has a rich history.  Tait and Little started
classifying alternating links using flypes, and they conjectured that
two reduced alternating diagrams represent the same link if and only
if they are related by flypes; i.e., one can be obtained from the
other by a sequence of flypes. A century later, Menasco and
Thistlethwaite \cite{menasco-thist:alternating} proved the ``Tait
Flyping Conjecture'' for all alternating links.

\begin{corollary}\label{cor:main}
Let $\L$ and $\L'$ be any biperiodic alternating links, such that their toroidally alternating $\Lambda$--quotient links $L$ and $L'$ are related by flypes. 
Then the determinant densities of $\L$ and $\L'$ are equal.
\end{corollary}
\begin{proof}
In the limit above, both $\det(K_n)$ and $c(K_n)$ are invariant under flypes.
\end{proof}

\section{Toroidal dimer model}\label{sec:dimer}
The study of the dimer model is an active research area in statistical
mechanics (see the excellent introductory lecture notes \cite{cimasoni-survey, kenyon-survey}).  
A {\em dimer covering} (or {\em perfect matching}) of a graph is a
subset of edges that cover every vertex exactly once; i.e., a pairing
of adjacent vertices.  The dimer model is the study of the set of
dimer coverings of $G$.  As we discuss below, it is also related to
the spanning tree model of an associated planar graph.

\subsection*{Planar graphs}
The simplest case is when $G$ is a finite balanced bipartite planar
graph, with edge weights $\mu_e$ for each edge $e$ in $G$.  A {\em
  Kasteleyn weighting} is a choice of sign for each edge, such that
each face of $G$ with $0$ mod $4$ edges has an odd number of signs,
and each face with $2$ mod $4$ edges has an even number of signs.  A
{\em Kasteleyn matrix} $\K$ is a weighted adjacency matrix of
$G$, such that rows are indexed by black vertices, and columns by
white vertices.  The matrix coefficients are $\pm\mu_e$, with the sign
determined by the Kasteleyn weighting.  Then, taking the sum over all
dimer coverings $M$ of $G$, the {\em partition function} $Z(G)$
satisfies (see \cite{cimasoni-survey, kenyon-survey}):
$$ Z(G):= \sum_M \prod_{e\in M} \mu_e  =|\det \K|.$$  
With $\mu_e=1$ for all edges, $Z(G)$ is the number of dimer coverings of $G$.  
Also see \cite{cdr-dimers} for relations between dimer coverings of planar graphs and knot theory.

\subsection*{Toroidal graphs}
Now, let $G$ be a finite balanced bipartite toroidal graph.  As in the
planar case, we choose a Kasteleyn weighting on edges of $G$.  We then
choose oriented simple closed curves $\gamma_z$ and $\gamma_w$ on
$T^2$, transverse to $G$, representing a basis of $H_1(T^2)$.  We
orient each edge $e$ of $G$ from its black vertex to its white vertex.
The weight on $e$ is
$$ \mu_e= z^{\gamma_z\cdot e}w^{\gamma_w\cdot e},$$ where $\cdot$
denotes the signed intersection number of $e$ with $\gamma_z$ or
$\gamma_w$.  For example, see Figure~\ref{square-fd1}.  The Kasteleyn
matrix $\K(z,w)$ is the weighted adjacency matrix with rows indexed by
black vertices and columns by white vertices, and matrix entries
$\pm\mu_e$, with the sign determined by the Kasteleyn weighting.  The
{\em characteristic polynomial} is defined as
$$ p(z,w) = \det \K(z,w).$$
See Section~\ref{sec:examples} for examples.  With $\mu_e$ as above, the number of dimer coverings of $G$ is given by (see \cite{cimasoni-survey, kenyon-survey}):
$$ Z(G) = \frac{1}{2} \ |-p(1,1) + p(-1,1) + p(1,-1) + p(-1,-1)|.$$

\begin{figure}

\begin{tikzpicture} 
\node at (0,0)
{ \includegraphics[height=1.75in]{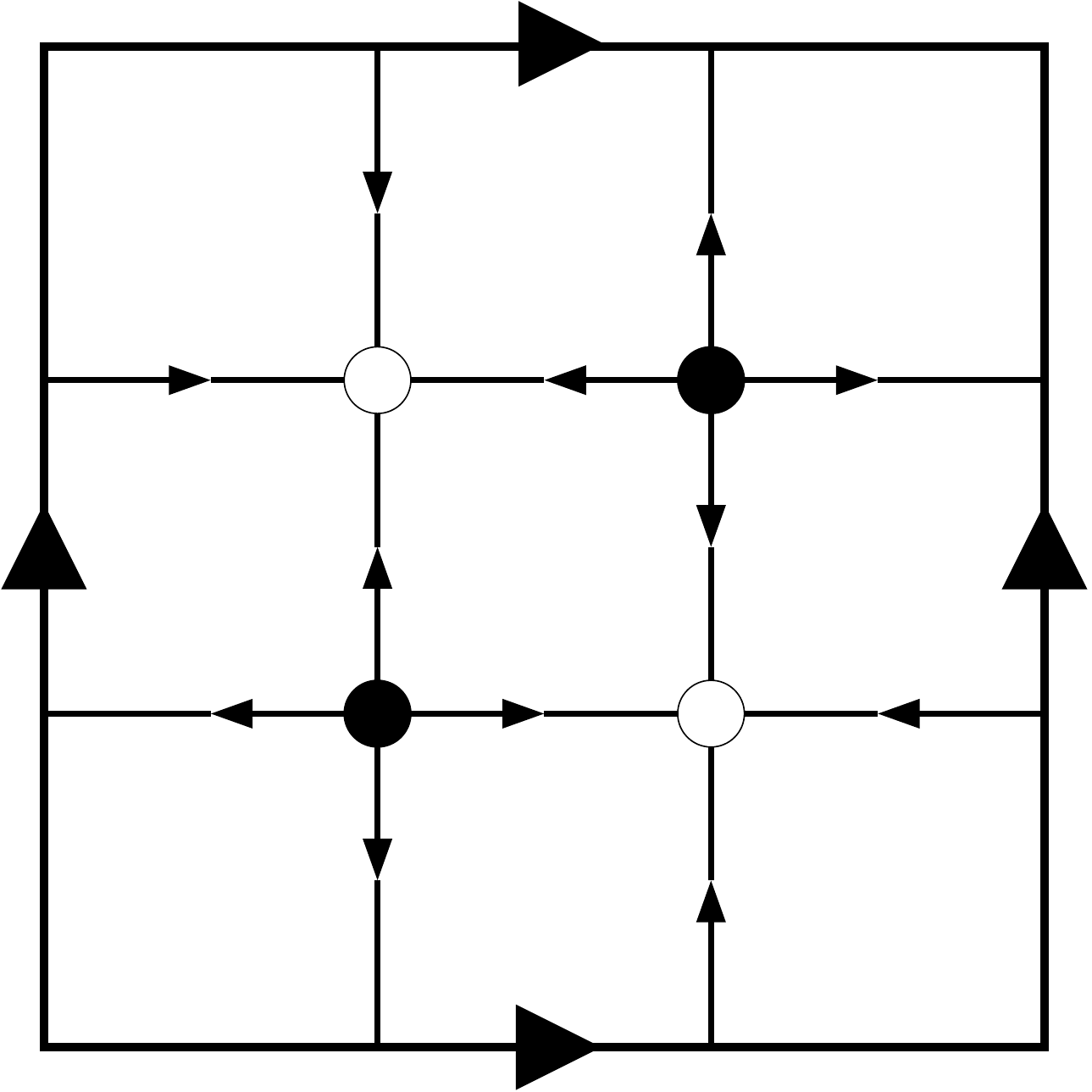}};
\node at (-2.5,0) {$\gamma_w$};
\node at (0,-2.5) {$\gamma_z$};
\node at (2.5,0) {$\gamma_w$};
\node at (0,2.5) {$\gamma_z$};
\node at (1,1) {2};
\node at (1,-1) {$2'$};
\node at (-1,1) {$1'$};
\node at (-1,-1) {1};
\node at (-0.3,0) {$\fbox{-}$};
\node at (-0.3,1.2) {$\fbox{-}$};

\draw[thick,->] (-6,0) -- (-4,0);
\draw[thick,->] (-5,-1) -- (-5,1);
\node at (-4,0.3) {$\gamma_z$};
\node at (-5.3,1) {$e$};
\node at (-5,-2) {$\mu_e=z$};

\draw[thick,->] (-10,0) -- (-8,0);
\draw[thick,<-] (-9,-1) -- (-9,1);
\node at (-8,0.3) {$\gamma_z$};
\node at (-9.3,1) {$e$};
\node at (-9,-2) {$\mu_e=\displaystyle{\frac{1}{z}}$};

\node at (-7,-3) {(a)};
\node at (0,-3) {(b)};

\end{tikzpicture}

\caption{(a) Edge weights $\mu_e= z^{\gamma_z\cdot e}$ to compute $\K(z,w)$. (b) Toroidal bipartite
  graph $G$ with a choice of Kasteleyn weighting.}
\label{square-fd1}
\end{figure}
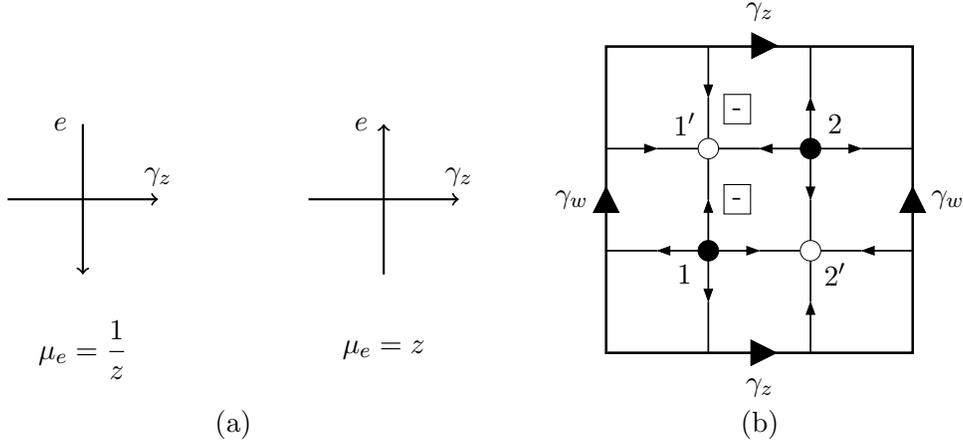

\subsection*{Biperiodic graphs}
Finally, let $G$ be a biperiodic bipartite planar graph $G$, so that
translations by a two-dimensional lattice $\Lambda$ act by
isomorphisms of $G$.  Let $G_n$ be the finite balanced bipartite
toroidal graph given by the quotient $G/(n\Lambda)$.  Kenyon, Okounkov
and Sheffield \cite{KOS} gave an explicit expression for the growth
rate of the toroidal dimer model on $\{G_n\}$:
\begin{theorem}\cite[Theorem~3.5]{KOS}\label{thm:KOS}
$$ \log Z(G) := \lim_{n\to\infty}\frac{1}{n^2}\log Z(G_n) = m(p(z,w)).$$
\end{theorem}
The quantity $\log Z(G)$ on the left is called the entropy of the
toroidal dimer model, or the partition function per fundamental
domain.  Theorem~\ref{thm:KOS} says that, independent of any choice of
Kasteleyn weighting and any choice of homology basis for the
$\Lambda$--action, the entropy of any toroidal dimer model is given by
the Mahler measure of its characteristic polynomial.

\section{Examples}\label{sec:examples}

\subsection{Square weave}

\begin{figure}[h] 
\begin{tabular}{cc}
 \includegraphics[height=1.3in]{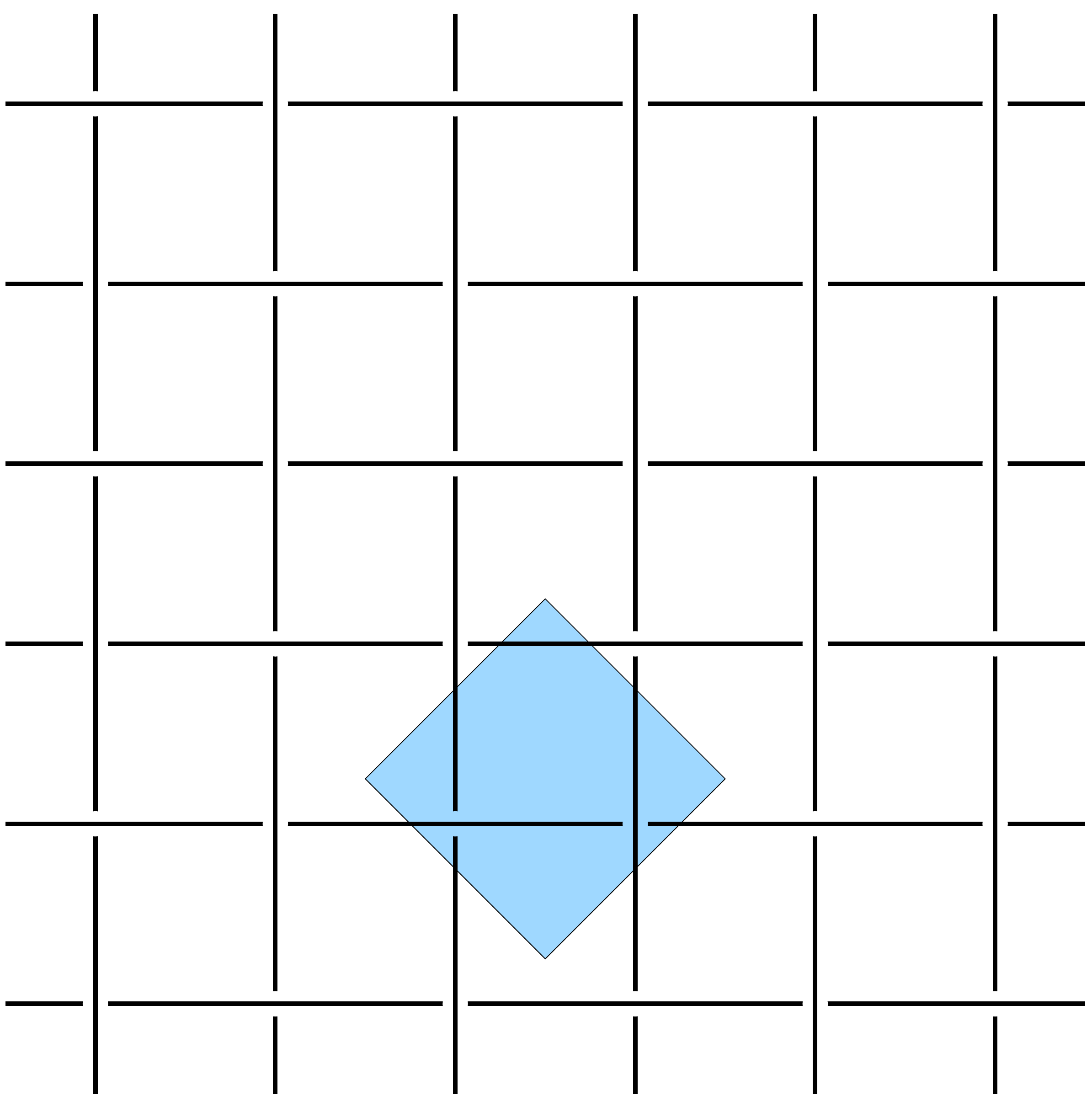} &  \quad \quad \quad \quad 
 \includegraphics[height=1.3in]{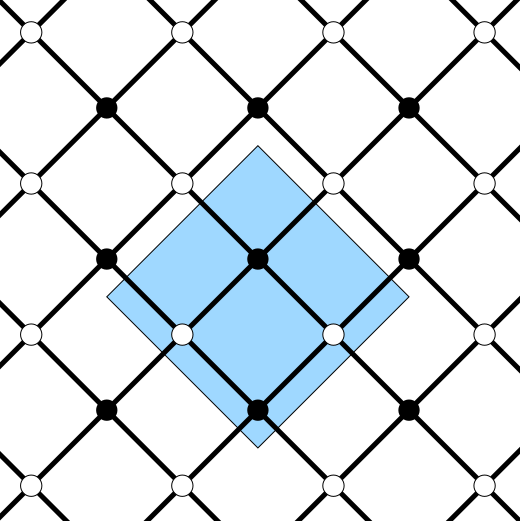}    \\
  (a) & \quad  \quad \quad \quad  (b)   \\
\end{tabular}
\caption{(a) Infinite square weave $\W$ and fundamental domain for $L$.\qquad
  (b) Bipartite graph $G_{\W}^b$ and fundamental domain for $G_L^b$.}
\label{fig-square}
\end{figure}

As mentioned in the introduction, Figure~\ref{fig:weave} shows the
infinite square weave $\W$.  Figure~\ref{fig-square}(a) shows a
slightly different fundamental domain from Figure~\ref{fig:weave}(b)
and the toroidally alternating link $L_1$ with $c(L_1)=2$.  Both of
the Tait graphs of $\W$ are the infinite square grid.  They are shown
overlaid in Figure~\ref{fig-square}(b), which shows the biperiodic
bipartite graph $G_{\W}^b$, as well as the fundamental domain for
$G_{L_1}^b$, which matches the toroidal graph shown in Figure
\ref{square-fd1}(b).

We now compute $p(z,w)= \det\K(z,w)$ for $G=G_{\W}^b$, as described
in Section~\ref{sec:dimer}, and in more detail in \cite{cimasoni-survey, kenyon-survey}.
Using Figure~\ref{square-fd1}(b) with the ordering as shown, 

\[ 
\K(z,w)=
\begin{bmatrix}
-1 - 1/z& 1 + w\cr
1+1/w& 1+z
\end{bmatrix},
\quad p(z,w) = -\left(4 + \frac{1}{w} + w + \frac{1}{z} +z\right)
\]

Exact computations (\cite{Boyd-specialvalues, RV-modular})
imply that
$\displaystyle 2\pi\, m(p(z,w)) =  2\,\voct$.
Therefore, by Theorems~\ref{thm:gmax} and \ref{thm:main},
$\voct$ is the limit of both determinant densities and volume densities for $K_n\toF\W$:
$$\lim_{n\to \infty} \frac{2 \pi \log \det(K_n)}{c(K_n)}
=\frac{2 \pi m( p(z,w))}{c(L_1)}
=\voct 
=\lim_{n\to \infty} \frac{\vol(K_n)}{c(K_n)}.
$$

\subsection{Triaxial link}

\begin{figure}[h]
\begin{tabular}{cc}
\includegraphics[height=1.5in]{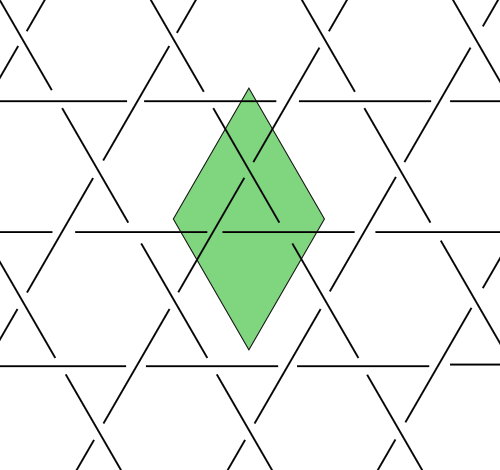} &
\hspace*{1cm}
 \includegraphics[height=1.5in]{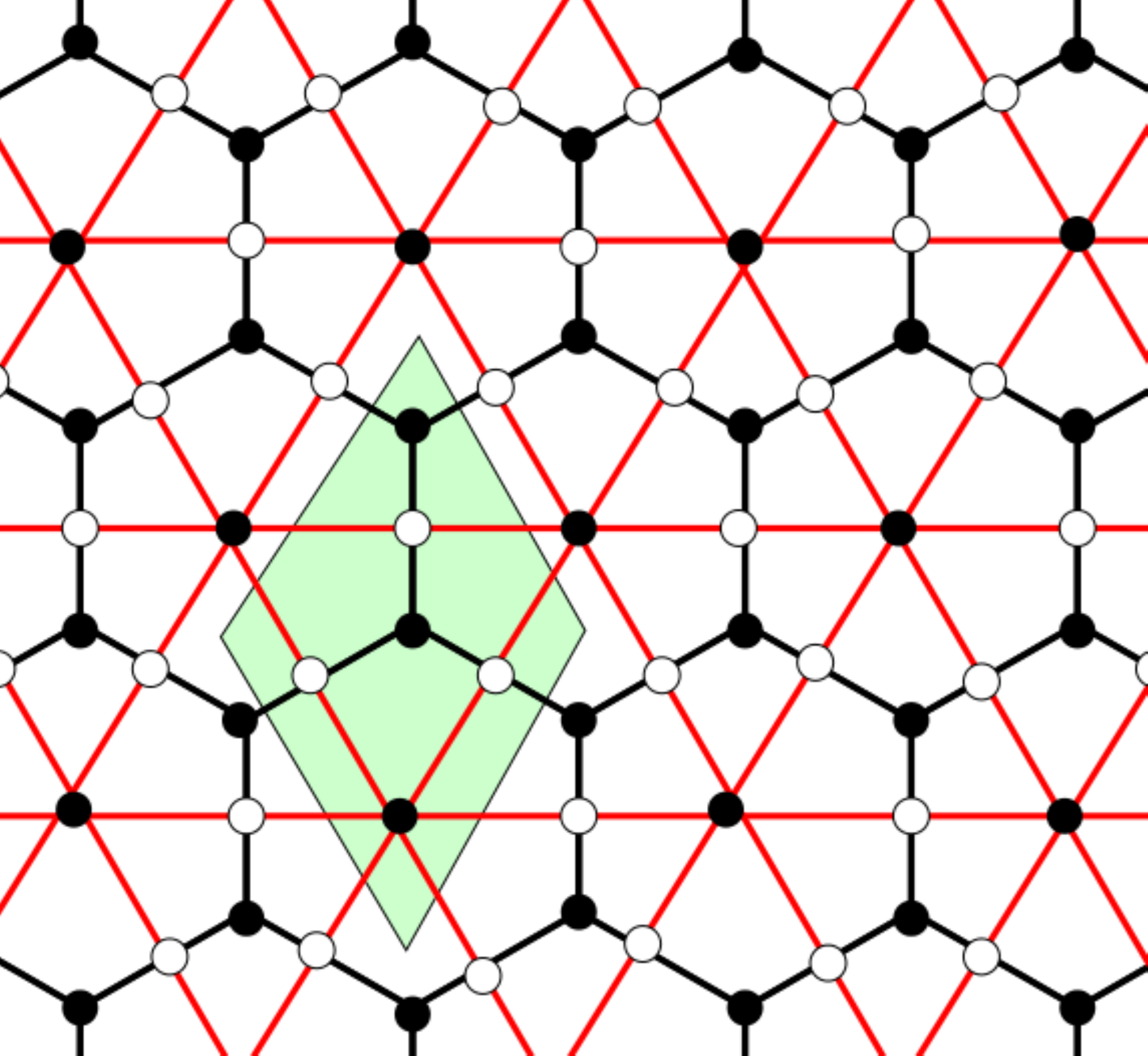} \\
(a) & \hspace*{1cm} (b)  \\
\end{tabular}
\caption{(a) Diagram of biperiodic triaxial link $\L$, and fundamental
  domain for $L$.  (b) Bipartite graph $G_{\L}^b$ and fundamental
  domain for $G_L^b$.}
\label{fig:triaxial}
\end{figure}

Figure~\ref{fig:triaxial}(a) shows part of the biperiodic
alternating diagram of the {\em triaxial link} $\L$, and the
fundamental domain for the toroidally alternating link $L$.  Its
projection graph $G(\L)$ is the trihexagonal tiling.  The Tait graphs
of $\L$ are the regular hexagonal and triangular tilings, shown
overlaid in Figure~\ref{fig:triaxial}(b) to form the biperiodic
balanced bipartite graph $G_{\L}^b$.

We now compute $p(z,w)= \det\K(z,w)$ for $G=G_{\L}^b$, as described
in Section~\ref{sec:dimer}, and in more detail in
\cite{cimasoni-survey, kenyon-survey}.  Using
Figure~\ref{fig:Kasteleyn}, with the homology basis, ordered vertices
and a choice of Kasteleyn weighting on edges as shown,

\begin{figure}

\begin{tikzpicture} 
\node at (0,0)
{ \includegraphics[height=2in]{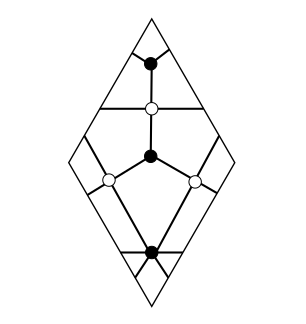}};
\node at (0,1.8) {\tiny{$1$}};
\node at (0.2,1.1) {\tiny{$1'$}};
\node at (0.25,0.25) {\tiny{$2$}};
\node at (-0.4,-0.4) {\tiny{$2'$}};
\node at (0.4,-0.4) {\tiny{$3'$}};
\node at (0,-1.1) {\tiny{$3$}};

\draw[thick,->] (-2,0) -- (-.95,1.8);
\node at (-2,1) {$\gamma_w$};
\draw[thick,->] (-2,0) -- (-.95,-1.8);
\node at (-2,-1) {$\gamma_z$};

\node at (-.25,-.65) {$-$};
\node at (-.4,.7) {$-$};
\node at (.7,.25) {$-$};

\end{tikzpicture}

\caption{For the triaxial link $\L$, the toroidal graph $G_L^b$, with
  a choice of homology basis, ordered vertices and a choice of
  Kasteleyn weighting on edges.}
\label{fig:Kasteleyn}
\end{figure}
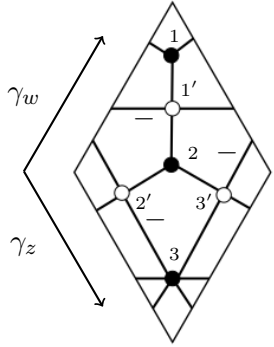

\[ 
\K(z,w)=
\begin{bmatrix}
1 & z & w \cr
1& 1& 1 \cr
1/z-1/w & 1/w-1 & 1-1/z 
\end{bmatrix},
\quad p(z,w)= 6-\left(\frac{1}{w}+w+\frac{1}{z}+z+\frac{w}{z}+\frac{z}{w}\right)
\]
Using exact computations (\cite{Boyd-specialvalues, RV-modular})
we can verify that $\displaystyle 2\pi\,m(p(z,w)) = 10\vtet$, where
$\vtet\approx 1.01494$ is the hyperbolic volume of the regular ideal
tetraheron.
Therefore, by Theorem~\ref{thm:main},
$$\lim_{n\to \infty} \frac{2 \pi \log \det(K_n)}{c(K_n)}
=\frac{2 \pi m( p(z,w))}{c(L)}
=\frac{10\vtet}{3}. $$
Moreover, in \cite{ckp:bal} we show that for the triaxial link $\L$,
$$ \lim_{n\to \infty} \frac{\vol(K_n)}{c(K_n)} = \frac{\vol(T^2 \times I - L)}{c(L)} = \frac{10\vtet}{3}.$$

Although for the square weave and the triaxial link, the volume and
determinant densities both converge to the volume density of the
toroidal link, this does not seem to be true in general.  We discuss
many such examples in \cite{ckp:bal}.

\section{Proof of Theorem \ref{thm:main}}\label{sec:proof}

Henceforth, let $G(\L)$ and $G_\L$ be the projection graph and Tait graph of $\L$, respectively, both of which are $\Lambda$--biperiodic.
Let $K_n$ be alternating link diagrams that F{\o}lner
converge almost everywhere to $\L$, so that $G(K_n)\toF G(\L)$.
Let $G_n\subset G(K_n)$ form the toroidal F{\o}lner sequence for $G(\L)$.
Note that $|G(K_n)|=c(K_n)$.

\begin{lemma}\label{lemma:Tait_folner}
As $K_n$ F{\o}lner converges almost everywhere to $\L$, the sequence of Tait graphs $G_{K_n}$ F{\o}lner converges almost everywhere to $G_\L$; i.e.,
$$ K_n\toF \L \quad \Longrightarrow \quad G_{K_n} \toF G_\L $$
\end{lemma}
\begin{proof}
The choice of checkerboard coloring of faces of $G(\L)$ to form $G_\L$
induces a checkerboard coloring of faces of $G(K_n)$, and hence a
unique choice of Tait graphs $G_{K_n}$.  We define $H_n\subset
G_{K_n}$ as follows: The edge $e\in G_{K_n}$ is an edge of $H_n$ if
and only if the corresponding vertex $v\in G(K_n)$ is a vertex of
$G_n-\partial G_n$.  Because $G_n$ is an exhaustive nested sequence of connected
subgraphs of $G(\L)\cap (n\Lambda)$, it follows immediately that $H_n$ is an
exhaustive nested sequence of connected subgraphs of $G_\L\cap (n\Lambda)$.

Since $\L$ is biperiodic, there is a positive integer $d$ such that
$\max(\deg(G(\L)), \deg(G_\L))\leq d$.  To prove the F{\o}lner
condition, let $v\in\partial H_n$.  As a vertex in $G_{\L}$, $v$ is
incident to a collection of edges $\{e_i\in H_n\}$ and $\{e'_j\notin
H_n\}$, with $0<i, j<d$.  By definition of $H_n$, these edges
correspond to vertices $\{v_i, v'_j \in G_n\}$, such that some
$v'_j\in \partial G_n$.  Hence, $|\partial H_n|\leq d\,|\partial
G_n|$.  Also by definition of $H_n$, $|G_n|\leq d\,|H_n|$.  Therefore,
$$ 0\leq \frac{|\partial H_n|}{|H_n|} \leq \frac{d\,|\partial G_n|}{\frac{1}{d}|G_n|} = d^2\frac{|\partial G_n|}{|G_n|}\underset{n\to\infty}{\longrightarrow} 0 $$

Let $v^{\out}_n$ and $e^{\out}_n$ be the numbers of vertices and edges, respectively, of $G_{K_n}-H_n$. 
To prove item ($ii$) of Definition \ref{def:folner_converge_graphs}, we will show $\lim\limits_{n\to\infty} v^{\out}_n/|G_{K_n}| = 0$.

\begin{eqnarray}
\label{eqn:edensity}
\nonumber 
 0\leq \frac{v^{\out}_n}{|G_{K_n}|}  &\leq&   \frac{e^{\out}_n}{|G_{K_n}|} \leq 
 \frac{|G(K_n)-(G_n-\partial G_n)|}{|H_n|} 
\leq \frac{|G(K_n)-(G_n-\partial G_n)|}{\frac{1}{d}|G_n|} \\ 
   &\leq& \frac{d \left(|G(K_n)|-|G_n|+|\partial G_n|\right)}{|G_n|} 
=
\frac{d \left(c(K_n)-|G_n|\right)}{c(K_n)}\cdot\frac{c(K_n)}{|G_n|} + \frac{d\;|\partial G_n|}{|G_n|}\underset{n\to\infty}{\longrightarrow} 0
\end{eqnarray}
The final limit follows from $\lim\limits_{n\to\infty} |G_n|/c(K_n) = 1$ and the F{\o}lner condition. 
\end{proof}

\begin{lemma}\label{lemma:ratio}
As $K_n$ F{\o}lner converges almost everywhere to $\L$, with $L$ its $\Lambda$--quotient link, 
$$\lim_{n\to\infty}\frac{e(G_{K_n})}{|G_{K_n}|} = \frac{e(G_L)}{|G_L|}.$$ 
\end{lemma}
\begin{proof}
Let $H_n\subset G_{K_n}$ form the toroidal F{\o}lner sequence for $G_\L$ as in Lemma~\ref{lemma:Tait_folner}.
Let $v^{\In}_n,\, v^{\out}_n,\, e^{\In}_n,\, e^{\out}_n$ be the numbers of vertices and
edges, respectively, of $H_n$ and $G_{K_n}-H_n$.  

First, we claim that 
\begin{equation}\label{eq:v-out}
\lim_{n\to\infty}\frac{v^{\out}_n}{e^{\In}_n}=0.
\end{equation}
To prove this claim, for every integer $k>0$, let $|f_n^k|$
denote the number of $k$--faces of $G(K_n)$ that are not contained in
$G_n$.  Hence, $v^{\out}_n=|G_{K_n}-H_n|\leq \sum_k|f_n^k|$.  By
counting vertices, $\sum_k k\,|f^k_n|\leq 4\,|G(K_n)-G_n|$.  The
factor $4$ appears because every vertex belongs to four faces, so it
will be counted at most four times in the sum.  Now, since
$e^{\In}_n=|G_n-\partial G_n|, \ |G(K_n)|=c(K_n)$, and using the
limits in Definition \ref{def:folner_converge}, we have
$$ 0\leq\frac{v^{\out}_n}{e^{\In}_n}\leq\frac{\sum_k|f_n^k|}{|G_n-\partial G_n|}\leq\frac{\sum_k k|f_n^k|}{|G_n-\partial G_n|}\leq\frac{4\,|G(K_n)-G_n|}{|G_n-\partial G_n|} \underset{n\to\infty}{\longrightarrow} 0.$$

We can now complete the proof of the lemma:
$$\lim_{n\to\infty} \frac{e(G_{K_n})}{|G_{K_n}|} = \lim_{n\to\infty}\frac{e^{\In}_n+e^{\out}_n}{|G_{K_n}|}=\lim_{n\to\infty}\frac{e^{\In}_n}{|G_{K_n}|} \quad \text{\rm by 
equation~(\ref{eqn:edensity})}.$$
$$\lim_{n\to\infty}\frac{|G_{K_n}|}{e^{\In}_n}=\lim_{n\to\infty}\frac{v^{\In}_n+v^{\out}_n}{e^{\In}_n}=\lim_{n\to\infty}\frac{v^{\In}_n}{e^{\In}_n} \quad \text{\rm by equation~(\ref{eq:v-out})}.$$
The biperiodicity of $\L$ implies that the final limit is the corresponding ratio for the $\Lambda$--quotient link $L$.
\end{proof}

\subsection*{Spanning trees and dimers} 

For any finite plane graph $G$, let $G^b$ be the balanced bipartite
graph as in Section~\ref{sec:definitions}: After overlaying $G$ and $G^*$, the black vertices of
$G^b$ are the vertices of $G$ and of $G^*$; the white vertices are
points of intersection of their edges.  To make $G^b$ balanced, we
then delete the vertex of $G^*$ corresponding to the unbounded face
and a vertex of $G$ adjacent to the unbounded face, along with all
incident edges to these vertices (see Figure
\ref{fig-spt-dimer}). Euler's formula implies that $G^b$ is a
balanced bipartite graph, and that $|G^b|= 2\,e(G)$.

By \cite{burton-pemantle, propp93}, the spanning trees of $G$ are in
bijection with the dimer coverings of $G^b$; i.e., $\tau(G)=Z(G^b)$.
Hence, for any alternating link $K$, we have
\begin{equation}
\label{eqn-ddd}
\frac{\log \det(K)}{c(K)}=\frac{\log \tau(G_K)}{e(G_K)} = \frac{\log Z(G_K^b)}{c(K)} 
\end{equation}

\begin{figure}[h] 
 \includegraphics[height=1.5in]{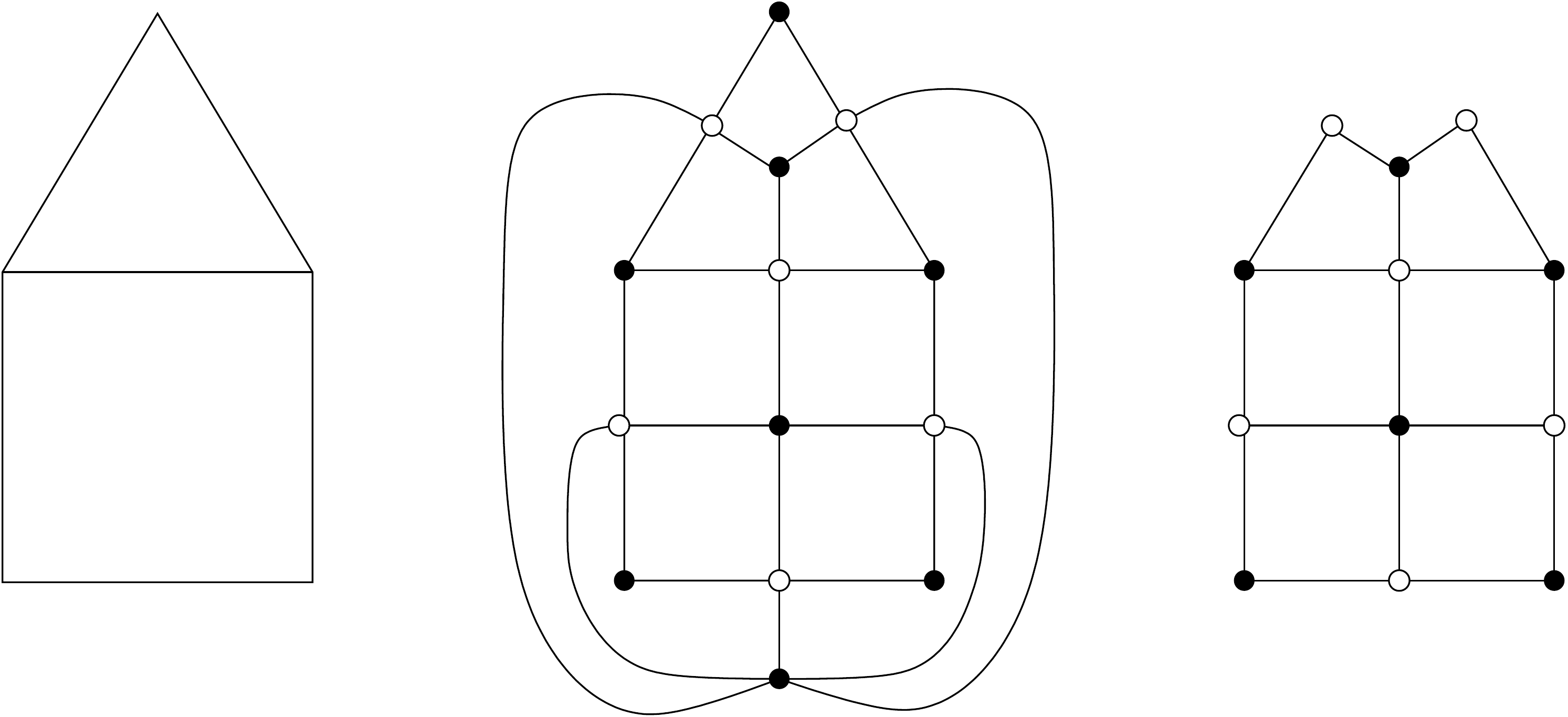}
 \caption{Graph $G$, overlaid graph $G \cup G^*$, balanced
   bipartite graph $G^b$ }
\label{fig-spt-dimer}
\end{figure}

Because $G(\L)$ is biperiodic and 4--valent, the overlaid graph
$G_{\L}^b=G_{\L}\cup G_{\L}^*$ is already a biperiodic balanced
bipartite graph.  So we can consider the toroidal dimer model on
$G_{\L}^b$.  We now relate the entropy of the toroidal dimer model on
$G_{\L}^b$ with the spanning tree entropy of $G_{\L}$.

If $H_n\toF G$, then the spanning tree entropy of $G$ is defined as 
$$ h(G) = \lim_{n\to\infty}\frac{\log\tau(H_n)}{|H_n|}. $$
In \cite{ckp:gmax}, the spanning tree entropy of the infinite square grid was used to prove the determinant density limit in Theorem~\ref{thm:gmax}.
For more details on spanning tree entropy and a broader context, see \cite{lyons}.
If $G$ is a biperiodic planar graph, we define the {\em normalized spanning tree entropy} as
$$ \tilde h(G) = \lim_{n\to\infty}\frac{\log\tau(H_n)}{n^2}.$$
Hence, if $G$ is $\Lambda$--biperiodic and $H_n = G \cap (n\Lambda)$, then $\tilde h(G) = |H_1|\,h(G)$.

\begin{prop}\label{prop:entropy}
Let $G$ be a $\Lambda$--biperiodic planar graph, and let $G^b=G \cup G^*$ be the
overlaid graph, which is a bipartite biperiodic planar graph.  Take
the natural exhaustions of $G^b$ by finite toroidal graphs $G^b_n = G^b/(n\Lambda)$,
and of $G$ by finite planar graphs $H_n = G \cap (n\Lambda)$.  Then as $n\to\infty$,
the entropy of the toroidal dimer model of $G^b$ equals the normalized spanning
tree entropy of $G$; i.e.,
$$ \log Z(G^b) = \lim_{n\to\infty}\frac{\log Z(G^b_n)}{n^2} = \lim_{n\to\infty}\frac{\log\tau(H_n)}{n^2} $$
\end{prop}
\begin{proof}
Temperley's bijection between spanning trees and dimers on the square
grid was extended by Burton and Pemantle \cite{burton-pemantle} to
general planar graphs, and further extended by Kenyon, Propp and
Wilson \cite{KPW} to directed weighted planar graphs.  Their main
result is that there is a measure-preserving bijection between
oriented weighted spanning trees of $H_n$ and dimer coverings on
$G^b_n$.  Hence, the claim follows in principle from \cite{KPW} which,
although not stated for infinite planar graphs, holds for large planar
graphs. The entropy in the infinite case can be obtained from a limit
of large planar graphs.
\footnote{We thank Richard Kenyon for this observation.}
\end{proof}

The following proposition shows that spanning tree entropy is not affected by peripheral changes in the graphs that converge as we defined above.
\begin{define}\cite[p. 498]{lyons}\label{def:tight}
Given $R>0$ and a finite graph $G$, let $\E_R(G)$ be the distribution
of the number of edges in the ball of radius $R$ about a random vertex
of $G$.  The sequence of graphs $G_n$ is {\em tight} if for
each $R$, the sequence of corresponding distributions $\E_R(G_n)$ is
tight.  In other words, for each $R>0$, 
$$ \limsup_{t\to\infty}\limsup_{n\to\infty}\,{\mathbb P}(\E_R(G_n)>t)=0. $$
\end{define}

\begin{prop}\cite[Corollary~3.8]{lyons}\label{prop:lyons}
Let $G_n$ be any tight sequence of finite connected graphs with bounded average degree such that $\displaystyle \lim_{n\to\infty}\frac{\log\tau(G_n)}{|G_n|} = h$.
If $H_n$ is a sequence of connected subgraphs of $G_n$ such that 
\begin{equation}\label{eq:lyons}
\lim_{n\to\infty}\frac{\# \{x \in V(H_n):\; \deg_{H_n}(x)=\deg_{G_n}(x)\} }{|G_n|} = 1,
\end{equation}
then \ 
$\displaystyle \lim_{n\to\infty}\frac{\log\tau(H_n)}{|H_n|} = h$.
\end{prop} 

\begin{proof}[{\bf Proof of Theorem \ref{thm:main}}]
By Lemma~\ref{lemma:Tait_folner}, $\displaystyle K_n\toF \L \ \Longrightarrow \ G_{K_n} \toF G_\L$.
Let $H_n\subset G_{K_n}$ form the toroidal F{\o}lner sequence for
$G_\L$, with the planar graphs $H_n\subset G_\L\cap(n\Lambda)$.
However, since $H_n$ is also exhaustive, we may assume $H_n =
G_\L\cap(n\Lambda)$; Proposition~\ref{prop:lyons} implies that the
spanning tree entropies are equal in either case.
The biperiodicity of $\L$ implies
\begin{equation}\label{eq:Hn}
\lim_{n\to\infty}\frac{|H_n|}{n^2\,|G_L|} = 1.
\end{equation}

In order to apply Proposition~\ref{prop:lyons} for the graphs
$H_n\subset G_{K_n}$, we now verify the required conditions.
Lemma~\ref{lemma:ratio} implies that $G_{K_n}$ have bounded average
degree.  Since 
 $H_n \subset G_\L$ which is biperiodic, 
for any $R>0$ there exists sufficiently large $t>0$ such that 
${\mathbb P}(\E_R(H_n)>t)=0$. Hence, 
the conditions in Definition~\ref{def:folner_converge_graphs} applied to
$H_n\subset G_{K_n}\toF G_\L$ imply that for any $R>0$ and sufficiently large $t>0$,
$$ 0\leq {\mathbb P}(\E_R(G_{K_n})>t)\leq \frac{|G_{K_n}-H_n|}{|G_{K_n}|} 
\underset{n\to \infty}{\longrightarrow} 0.$$
Thus, $G_{K_n}$ are a tight sequence of graphs.
Finally, the conditions in Definition~\ref{def:folner_converge_graphs} imply the limit (\ref{eq:lyons}).
Thus, by Proposition~\ref{prop:lyons} the spanning tree entropies for $H_n$ and $G_{K_n}$ are equal.

Let $G^b_n=G^b_\L/(n\Lambda)$, which is a balanced bipartite toroidal
graph.  Recall $e(G_K)=c(K)$.  By the results above, we have
\begin{eqnarray*}
\lim_{n\to\infty}\frac{\log\det(K_n)}{c(K_n)} & = & \lim_{n\to\infty}\frac{\log\tau(G_{K_n})}{e(G_{K_n})} \\ \\
 & = &  \lim_{n\to\infty}\frac{|G_{K_n}|}{e(G_{K_n})}\cdot\frac{\log\tau(G_{K_n})}{|G_{K_n}|} \\ \\
 & = &  \frac{|G_L|}{e(G_L)}\ \lim_{n\to\infty}\frac{\log\tau(G_{K_n})}{|G_{K_n}|} \hspace*{0.3in} \text{(by Lemma~\ref{lemma:ratio})} \\ \\
 & = &  \frac{|G_L|}{e(G_L)}\ \lim_{n\to\infty}\frac{\log\tau(H_n)}{|H_n|} \hspace*{0.4in} \text{(by Proposition~\ref{prop:lyons})} \\ \\
 & = &  \frac{1}{e(G_L)}\ \lim_{n\to\infty}\frac{\log\tau(H_n)}{n^2}\cdot \frac{n^2\,|G_L|}{|H_n|}  \\ \\
 & = &  \frac{1}{e(G_L)}\ \lim_{n\to\infty}\frac{\log\tau(H_n)}{n^2} \hspace*{0.4in} \text{(by equation~(\ref{eq:Hn}))} \\ \\
 & = &  \frac{1}{e(G_L)}\ \lim_{n\to\infty}\frac{\log Z(G^b_n)}{n^2} \hspace*{0.38in} \text{(by Proposition~\ref{prop:entropy})} \\ \\
 & = & \frac{m(p(z,w))}{c(L)} \hspace*{1.2in} \text{(by Theorem~\ref{thm:KOS})}. 
\end{eqnarray*}
\end{proof}

\subsection*{Acknowledgements}
We thank Richard Kenyon, Jessica Purcell, Neal Stoltzfus and
B\'eatrice de Tili\`ere for useful conversations.  We thank the
anonymous referee for careful revisions and thoughtful suggestions.
The authors acknowledge support by the Simons Foundation and PSC-CUNY.
The first author also thanks the Columbia University Mathematics
Department for its hospitality during his sabbatical leave.

\bibliographystyle{amsplain}
\bibliography{references}

\providecommand{\bysame}{\leavevmode\hbox to3em{\hrulefill}\thinspace}
\providecommand{\MR}{\relax\ifhmode\unskip\space\fi MR }
\providecommand{\MRhref}[2]{%
  \href{http://www.ams.org/mathscinet-getitem?mr=#1}{#2}
}
\providecommand{\href}[2]{#2}
\begin{thebibliography}{10}

\bibitem{Adams:crossings}
Colin Adams, Thomas Fleming, Michael Levin, and Ari~M. Turner, \emph{Crossing
  number of alternating knots in {$S\times I$}}, Pacific J. Math. \textbf{203}
  (2002), no.~1, 1--22.

\bibitem{BoydRV2002}
D.~Boyd and F.~Rodriguez-Villegas, \emph{Mahler's measure and the dilogarithm.
  {I}}, Canad. J. Math. \textbf{54} (2002), no.~3, 468--492.

\bibitem{Boyd-specialvalues}
David~W. Boyd, \emph{Mahler's measure and special values of {$L$}-functions},
  Experiment. Math. \textbf{7} (1998), no.~1, 37--82. \MR{1618282 (99d:11070)}

\bibitem{boyd-manifolds}
\bysame, \emph{Mahler's measure and invariants of hyperbolic manifolds}, Number
  theory for the millennium, {I} ({U}rbana, {IL}, 2000), A K Peters, Natick,
  MA, 2002, pp.~127--143. \MR{1956222}

\bibitem{burton-pemantle}
Robert Burton and Robin Pemantle, \emph{Local characteristics, entropy and
  limit theorems for spanning trees and domino tilings via
  transfer-impedances}, Ann. Probab. \textbf{21} (1993), no.~3, 1329--1371.
  \MR{1235419 (94m:60019)}

\bibitem{ckp:gmax}
Abhijit Champanerkar, Ilya Kofman, and Jessica~S. Purcell, \emph{Geometrically
  and diagrammatically maximal knots}, {\rm to appear in} J. London Math. Soc. 
(arXiv:1411.7915 [math.GT], 2015).

\bibitem{ckp:weaving}
\bysame, \emph{Volume bounds for weaving knots}, {\rm to appear in} Algebr.
  Geom. Topol. (arXiv:1506.04139 [math.GT], 2015).

\bibitem{ckp:bal}
\bysame, \emph{{Geometry of semi-regular biperiodic alternating links}}, in
  preparation.

\bibitem{cimasoni-survey}
David Cimasoni, \emph{The geometry of dimer models}, arXiv:1409.4631 [math-ph],
  2014.

\bibitem{cdr-dimers}
Moshe Cohen, Oliver~T. Dasbach, and Heather~M. Russell, \emph{A twisted dimer
  model for knots}, Fund. Math. \textbf{225} (2014), no.~1, 57--74.
  \MR{3205565}

\bibitem{crowell}
Richard~H. Crowell, \emph{Nonalternating links}, Illinois J. Math. \textbf{3}
  (1959), 101--120. \MR{0099667 (20 \#6105)}

\bibitem{KOS}
R.~Kenyon, A.~Okounkov, and S.~Sheffield, \emph{Dimers and amoebae}, Ann. of
  Math. (2) \textbf{163} (2006), no.~3, 1019--1056.

\bibitem{kenyon-survey}
Richard Kenyon, \emph{Lectures on dimers}, Statistical mechanics, IAS/Park City
  Math. Ser., vol.~16, Amer. Math. Soc., Providence, RI, 2009, pp.~191--230.
  \MR{2523460 (2010j:82023)}

\bibitem{KPW}
Richard~W. Kenyon, James~G. Propp, and David~B. Wilson, \emph{Trees and
  matchings}, Electron. J. Combin. \textbf{7} (2000), Research Paper 25, 34 pp.
  (electronic).

\bibitem{lickorish}
W.~B.~Raymond Lickorish, \emph{An introduction to knot theory}, Graduate Texts
  in Mathematics, vol. 175, Springer-Verlag, New York, 1997. \MR{1472978}

\bibitem{lyons}
Russell Lyons, \emph{Asymptotic enumeration of spanning trees}, Combin. Probab.
  Comput. \textbf{14} (2005), no.~4, 491--522. \MR{2160416 (2006j:05048)}

\bibitem{menasco-thist:alternating}
William Menasco and Morwen Thistlethwaite, \emph{The classification of
  alternating links}, Ann. of Math. (2) \textbf{138} (1993), no.~1, 113--171.
  \MR{1230928 (95g:57015)}

\bibitem{propp93}
James Propp, \emph{Lattice structure for orientations of graphs},
  arXiv:math/0209005v1 [math.CO].

\bibitem{sw-spanningtrees}
Daniel Silver and Susan Williams, \emph{Spanning trees and {M}ahler measure},
  arXiv:1602.02797 [math.GT], 2016.

\bibitem{smyth-mm-several}
C.~J. Smyth, \emph{On measures of polynomials in several variables}, Bull.
  Austral. Math. Soc. \textbf{23} (1981), no.~1, 49--63. \MR{615132}

\bibitem{smyth-survey}
Chris Smyth, \emph{The {M}ahler measure of algebraic numbers: a survey}, Number
  theory and polynomials, London Math. Soc. Lecture Note Ser., vol. 352,
  Cambridge Univ. Press, Cambridge, 2008, pp.~322--349. \MR{2428530}

\bibitem{RV-modular}
F.~Rodriguez Villegas, \emph{Modular {M}ahler measures. {I}}, Topics in number
  theory ({U}niversity {P}ark, {PA}, 1997), Math. Appl., vol. 467, Kluwer Acad.
  Publ., Dordrecht, 1999, pp.~17--48. \MR{1691309 (2000e:11085)}

\bibitem{flype_fig}
Wikipedia, https://commons.wikimedia.org/wiki/File:Flype.svg.

\end{thebibliography}
\end{document}